\documentclass[letterpaper, 10 pt, conference]{ieeeconf}
\IEEEoverridecommandlockouts
\usepackage{amsmath,amssymb,amsfonts}
\usepackage{graphicx}
\usepackage{subfigure}
\usepackage{epsfig} 

\usepackage[utf8]{inputenc}
\usepackage{hyperref} 
\usepackage{algorithm}
\usepackage{algpseudocode}

\algrenewcommand\algorithmicrequire{\textbf{Input:}}
\algrenewcommand\algorithmicensure{\textbf{Output:}}

\usepackage{xcolor}

\newtheorem{theorem}{Theorem}[section]
\newtheorem{lemma}[theorem]{Lemma}

\newcommand{\R}{\mathbb{R}}

\title{Fast-and-flexible decision-making with modulatory interactions}

\author{Rodrigo Moreno-Morton$^{1}$, Anastasia Bizyaeva$^2$, Naomi Ehrich Leonard$^3$, Alessio Franci$^{4}$%
\thanks{$^1$Department of Electrical Engineering and Computer Science, University of Liege, Belgium, {\tt\small luiromormor@gmail.com}}%
\thanks{$^2$NSF AI Institute in Dynamic Systems and Department of Mechanical Engineering at the University of Washington, Seattle, WA, 98195 USA, {\tt\small anabiz@uw.edu}}%
\thanks{$^3$Department of Mechanical and Aerospace Engineering, Princeton University, Princeton, NJ 08544 USA, {\tt\small naomi@princeton.edu}}%
\thanks{$^4$Department of Electrical Engineering and Computer Science, University of Liege, Belgium, and WEL Research Institute, Wavre, Beligum, {\tt\small afranci@uliege.be}}%
}

\begin{document}



\maketitle

\begin{abstract}

    Multi-agent systems in biology, society, and engineering are capable of making decisions through the dynamic interaction of their elements. Nonlinearity of the interactions is key for the speed, robustness, and flexibility of multi-agent decision-making. In this work we introduce modulatory, that is, multiplicative, in contrast to additive, interactions in a nonlinear opinion dynamics model of fast-and-flexible decision-making. The original model is nonlinear because network interactions, although additive, are saturated. Modulatory interactions introduce an extra source of nonlinearity that greatly enriches the model decision-making behavior in a mathematically tractable way.
    Modulatory interactions are widespread in both biological and social decision-making networks; our model provides new tools to understand the role of these interactions in networked decision-making and to engineer them in artificial systems. 
    
\end{abstract}

\section{Introduction}


We recently introduced a general, mathematically tractable, nonlinear opinion dynamics (NOD) model of fast-and-flexible multi-agent, multi-option decision-making~\cite{bizyaeva2022nonlinear,leonard2024fast,bizyaeva2023multi}. The model consists of a network of linear, first-order, stable dynamics with saturated network interactions and exogenous inputs. It is closely related to both bio-inspired~\cite{wilson1972excitatory} and artificial~\cite{hopfield1984neurons} recurrent neural networks models. It is also reminiscent of continuous-time models of gene regulatory networks~\cite{le2015quantitative} with linear degradation dynamics and saturated Hill-type~\cite{hill1910possible} molecular interactions. It has been used to model and analyze biological~\cite{gray2018multiagent} and sociopolitical~\cite{leonard2021nonlinear,bizyaeva2023multi} collective decision-making as well as to engineer a number of decision-making behaviors~\cite{amorim2023threshold,franci2021analysis,cathcart2023proactive,hu2023emergent,park2021tuning}.

In our NOD, the linear stable dynamics model negative feedback regulation of opinions towards a neutral, non-opinionated state. Saturated network interactions model nonlinear opinion exchanges that can amplify the information brought by exogenous inputs through positive feedback and trigger fast and strong opinion formation. The balance between negative and positive feedback is tuned by an {\it attention} parameter that models the agents' level of engagement in the decision-making process. When positive and negative feedback are perfectly balanced, the model undergoes a bifurcation at which the non-opinionated state becomes unstable and strong opinions are formed. This opinion-forming bifurcation is the key determinant of the fast-and-flexible nature of decision-making described by the NOD model~\cite{leonard2024fast}. 

In the original NOD model~\cite{bizyaeva2022nonlinear,leonard2024fast,bizyaeva2023multi} network interactions are saturated but otherwise {\it additive}. We provide ample evidence in Section~\ref{sec: modulation} that neural, biological, and sociopolitical decision-making networks also use {\it modulatory}, that is, {\it multiplicative}, interactions to dynamically modulate the strength of the additive network links in a distributed way. The NOD model in~\cite{bizyaeva2022nonlinear,leonard2024fast,bizyaeva2023multi} cannot systematically represent the effect of modulatory interactions.
Here, we extend this NOD model to include modulatory network interactions inspired by modulatory interactions observed in nature and society. We analyze how opinion-forming bifurcations are shaped by these interactions, and explore their use to engineer more complex decision-making behavior in autonomous agents in a mathematically tractable way.

The paper contributions are the following. {\it i)} We define a mathematically tractable parametrization of modulatory interactions in the NOD model introduced in~\cite{bizyaeva2022nonlinear,leonard2024fast,bizyaeva2023multi}. {\it ii)} We analyze how multiplicative interactions shape the model opinion-forming bifurcation behavior and we provide a thorough interpretion of the analytical results in terms of fast-and-flexible decision-making. {\it iii)} We use modulatory interactions to augment a recently introduced NOD-based robotic obstacle avoidance controller~\cite{cathcart2023proactive} and provide it with fast-and-flexible {\it conditional} decision-making capabilities.


The paper is organized as follows. 
In Section~\ref{sec: modulation} we review evidence of modulatory interactions in nature and society, and discuss their potential in engineered systems. In Section~\ref{sec: general model definitions} we define the general nonlinear opinion dynamics model with modulatory interactions and discuss its interpretation and possible generalizations. In Section~\ref{sec: bidimensional} we illustrate the main ideas and results of the paper on a representative example. Section~\ref{sec: analysis} presents the main analytical results of the paper and thoroughly discusses their consequence for the mathematical tractability of the model. Section~\ref{sec: shaping hierarchical} applies the main results of the paper to analyze how modulatory interactions shape opinion formation in a multi-agent two-option NOD model of modulatory social influence and in a single-agent multi-option NOD model of robotic navigation.

\section{Modulatory interactions in biological, social, and artificial decision-making networks}
\label{sec: modulation}

We describe several examples of the occurrence and relevance of modulatory interactions in biological and social system, and propose their implementation in artificial systems by suitably augmenting NOD.\\
\underline{Neural Networks.}
Recent evidence~\cite{boahen2022dendrocentric,murphy2016global,silver2010neuronal} suggests that multiplicative interactions between neurons (Figure~\ref{fig: modulation}, top left), implemented through nonlinear dendrites and neuromodulation, play a key role in many kinds of neural computation.
A fundamental building block of artificial neural networks, the Gated Recurrent Unit (GRU), is also defined by multiplicative interactions~\cite{krishnamurthy2022theory}.

\begin{figure}
\vspace{7pt}
    \centering
    \includegraphics[width=0.4\textwidth]{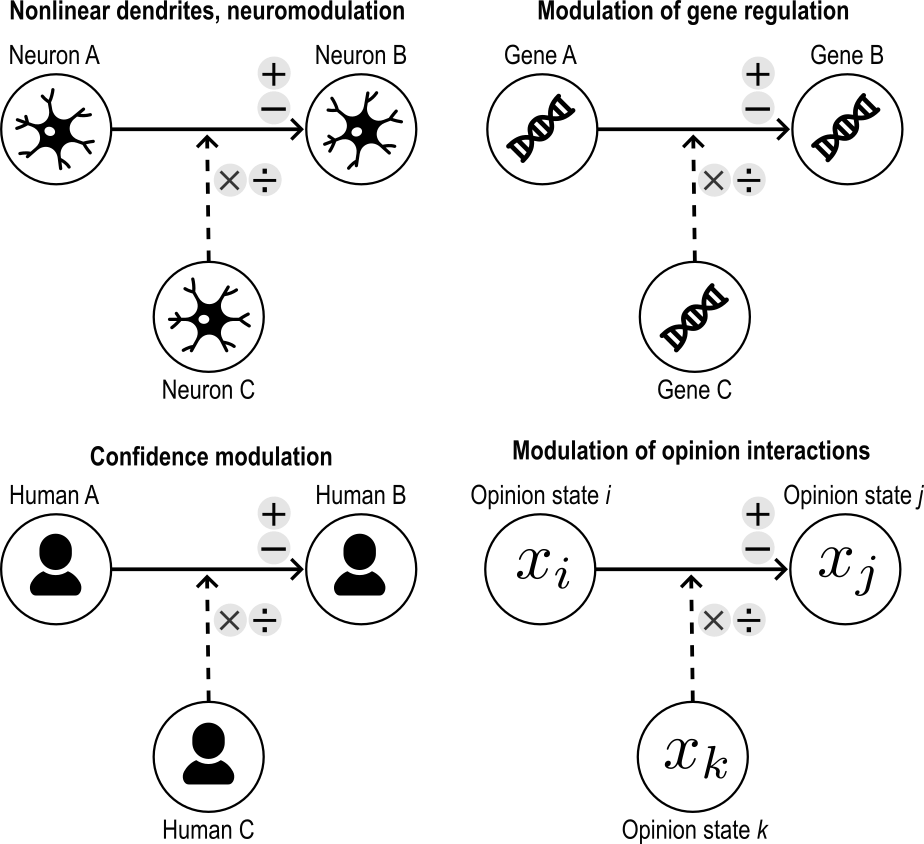}
    \caption{{\bf Top left}: Neuron A excites or inhibits Neuron B through additive synaptic inputs. Neuron C modulates the strength of this interaction through multiplicative dendritic and neuromodulatory inputs. {\bf Top right}: Gene A regulates the expression of Gene B through transcription factors. Gene C modulates the strength of this regulation. {\bf Bottom left}: Human A influences the opinion of Human B through verbal information exchange. Human C changes the strength of this influence by modulating the confidence that B has on A. {\bf Bottom right}: In nonlinear opinion dynamics, an opinion state $x_i$ has either a positive or negative effect
    on another opinion state $x_j$ through saturated additive interactions. Introducing multiplicative interactions, a third opinion state $x_k$ modulates how strong these effects are.
    }
    \label{fig: modulation}
\end{figure}

\underline{Gene Regulatory Networks.}
The modulation of the gene transcription machinery, and its effects on gene expression regulation, have been studied both in prokaryotes~\cite{iyerTranscriptionalRegulationFactor2020} and eukaryotes~\cite{sinhaHistoneModificationsRegulate2023}.
This kind of molecular modulatory interactions can naturally be modeled as multiplicative interactions in gene regulatory network models~\cite{le2015quantitative} (Figure~\ref{fig: modulation}, top right).

\underline{Social networks.} Modulatory interactions between social agents are often introduced into basic social network models, e.g., the DeGroot model~\cite{degroot1974reaching}, to reproduce and understand complex social phenomena 
(Figure~\ref{fig: modulation}, bottom left). Examples are biased assimilation models~\cite{dandekar2013biased,xia2020analysis}, bounded confidence models~\cite{brooks2022emergence}, social network models of belief polarization in Twitter discussions~\cite{baumann2020modeling}, models of asymmetric political polarization in the United States Congress \cite{leonard2021nonlinear}, epidemics model with risk aversion~\cite{bizyaeva2024active}, and war-peace transitions in neighboring nations \cite{morrison2023transitions}.

\underline{NOD for autonomous agent control.}
The nonlinear opinion dynamics in~\cite{bizyaeva2022nonlinear, leonard2024fast, bizyaeva2023multi} have recently been used for the control of autonomous agents, including self-driving cars~\cite{hu2023emergent} and robotic obstacle avoidance~\cite{cathcart2023proactive}. Inspired by the widespread occurrence of modulatory interactions in biological and social systems, in Section~\ref{sec: hierarchical decisions} we show how the NOD with modulatory interactions introduced in the next section (Figure~\ref{fig: modulation}, bottom right) provide the means to enable conditional fast-and-flexible decision-making in the NOD model proposed in~\cite{cathcart2023proactive} for robot navigation.



\section{A tractable NOD model with modulatory interactions}\label{sec: general model definitions}

The proposed NOD with modulatory interactions for a group of agents forming opinions about two options reads
\begin{align}\label{eq: general} 
    \dot{x}_i &= -x_i + b_i + S\left( \sum_{j=1}^N a_{ij} \left( u_0 + \sum_{k=1}^N m_{ijk} x_k^n \right) x_j \right),
\end{align}
$i=1,\ldots,N$, which can also be written in vector form as
\begin{equation}
    \dot{\mathbf{x}} = -\mathbf{x} + \mathbf{b} + {\bf S}( (u_0+\tilde M({\bf x}))\odot A \cdot \mathbf{x}) \label{eq: general vector}
\end{equation}
where ${\bf x}=[x_1\ \cdots x_N]^T \in \R^N$ represents the agents' opinion states, $x_i > 0$ (resp. $x_i < 0$) means a preference for option 1 (resp. option 2), $\mathbf{b}=[b_1\ \cdots\ b_N]^T \in \R^N$ are exogenous inputs, and $n \in \mathbb{N}_{>0}$ is the \textit{order} of the modulatory interaction. $A = [a_{ij}], i,j=1,\ldots,N$ is a matrix of interaction weights that determines the additive effect of the opinion of agent $j$ on the opinion of agent $i$. The term $u_0 + \tilde{M}(\mathbf{x})$
is the sum of the \textit{basal attention} $u_0$ that an agent is paying to its neighbors' opinions and the effects $\tilde M(\mathbf{x})=[\tilde m_{ij}(\mathbf{x})]$, $i,j=1,\ldots,N$, where
$
\tilde{m}_{ij}(\mathbf{x}) = \sum_{k=1}^N m_{ijk} x_k^n
$, that {\it modulatory interactions} have on the attention paid to specific neighbors. In particular, $m_{ijk}$ determines the sign and strength of the modulatory effect that the opinion of agent $k$ has on $a_{ij}$. Modulatory interactions generalize the state-dependent attention mechanism introduced in~\cite{bizyaeva2022nonlinear, bizyaeva2021control, franci2021analysis}.
Finally, $\mathbf{S}(\mathbf{x}) = [S(x_1),\cdots S(x_N)]$, with $S: \R \rightarrow \R$, is a vector of smooth saturation functions. To model symmetry between option 1 and option 2 in the absence of inputs, it is natural to assume that $S$ is odd-symmetric. Here, we simply assume $S(\cdot) = \tanh(\cdot)$.

Alternatively, model~\eqref{eq: general} can be interpreted as a modulated NOD for a single agent forming opinions about $N$ options. In this case, $x_i$ is the agent's opinion about option $i$, where $x_i>0$ ($x_i<0$) means that the agent favors (disfavors) option $i$, $a_{ij}$ is the weight with which the agent's opinion about option $i$ additively affects its opinion on option $j$, and $m_{ijk}$ is the weight with which the agent's opinion about option $k$ modulates $a_{ij}$. In the single-agent, multi-option interpretation of~\eqref{eq: general}, $S$ does not need to be odd-symmetric because there is no natural symmetry between favoring or disfavoring an option. Here, we simply assume $S$ is a shifted $\tanh$ function, i.e., $S(\cdot) = \frac{\tanh(\cdot-s)+\tanh(s)}{1-\tanh(s)^2}$, $s\in\R$.

Similarly to the original, non-modulated NOD model, the magnitude of the basal attention $u_0$ tunes the balance between the negative feedback regulation provided by the linear term in~\eqref{eq: general} and the positive feedback amplification provided by the saturated networked term. For sufficiently large basal attention, positive feedback dominates, which destabilizes the neutral state ${\bf x}={\bf 0}$ in an opinion-forming bifurcation that organizes the fast-and-flexible decision-making behavior described by NOD.


As we shall prove in Section~\ref{sec: analysis}, model~\eqref{eq: general} is tractable because  modulatory interactions do not affect {\it i)} the location of the bifurcation point in which the neutral state ${\bf x}={\bf 0}$ loses stability, i.e., the opinion-forming bifurcation thoroughly analyzed in~\cite{bizyaeva2022nonlinear,leonard2024fast,bizyaeva2023multi}, and {\it ii)} the kernel of the Jacobian of model~\eqref{eq: general} at bifurcation, which determines (to leading order) the opinion pattern observed once indecision is broken and the direction in the input space to which the system is most sensitive close to bifurcation. On the other hand, modulatory interactions affect the shape of the opinion-forming bifurcation branches. This non-local effect is hard to characterize in full generality but it can be analyzed on a case-by-case basis using Lyapunov-Schmidt reduction and singularity theory methods~\cite{Golubitsky1985}.


\subsection{Extensions and generalizations}
\label{sec: ext and gen}

Similarly to the original model~\cite{bizyaeva2022nonlinear,leonard2024fast}, we can extend~\eqref{eq: general} to the $N_a$-agent, $N_o$-option case.
This extension requires a six-index modulatory interaction matrix $M=[m_{ikk'}^{jll'}]$, where $m_{ikk'}^{jll'}$ is the effect of the opinion of agent $k'$ about option $l'$ on the additive interaction weight with which the opinion of agent $k$ about option $l$ affects the opinion of agent $i$ about option $j$.
As another generalization, replacing the modulatory term $m_{ijk} x_k^n$ by a general polynomial in $x_k$ would lead to a variant of~\eqref{eq: general} reminiscent of polynomial network dynamics on hypergraphs~\cite{pickard2023kronecker}


\section{An illustrative example}
\label{sec: bidimensional}

For two agents (or two options), model~\eqref{eq: general vector} reduces to
\begin{subequations}
\label{eq: bidimensional}
\begin{align}
    \tau \dot{x}_1 &= -x_1 + b_1 + S\Big( ( u_0 + m_{111}x_1^n + m_{112}x_2^n) a_{11} x_1+ \nonumber\\
    & + (u_0 + m_{121}x_1^n + m_{122}x_2^n) a_{12} x_2) \Big )\\
    \tau \dot{x}_2 &= -x_2 + b_2 + S\Big( ( u_0 + m_{211}x_1^n + m_{212}x_2^n) a_{21} x_1+ \nonumber\\
    & + (u_0 + m_{221}x_1^n + m_{222}x_2^n) a_{22} x_2) \Big )
\end{align}
\end{subequations}
For the modulated interaction network sketched in Figure~\ref{fig: toy}A, the additive and modulatory interaction matrices read
\begin{equation}\label{eq: bidimensional matrices}
    A=\begin{pmatrix}
    0 & -1 \\
    -1 & 0
\end{pmatrix},\ M_{\cdot\,\cdot\, 1}=\begin{pmatrix}
    0 & 0\\
    1 & 0
\end{pmatrix},\ M_{\cdot\,\cdot\, 2}=\begin{pmatrix}
    0 & 0\\
    0 & 0
\end{pmatrix}.
\end{equation}

\begin{figure}
\vspace{7pt}
    \centering
    \includegraphics[width=0.4\textwidth]{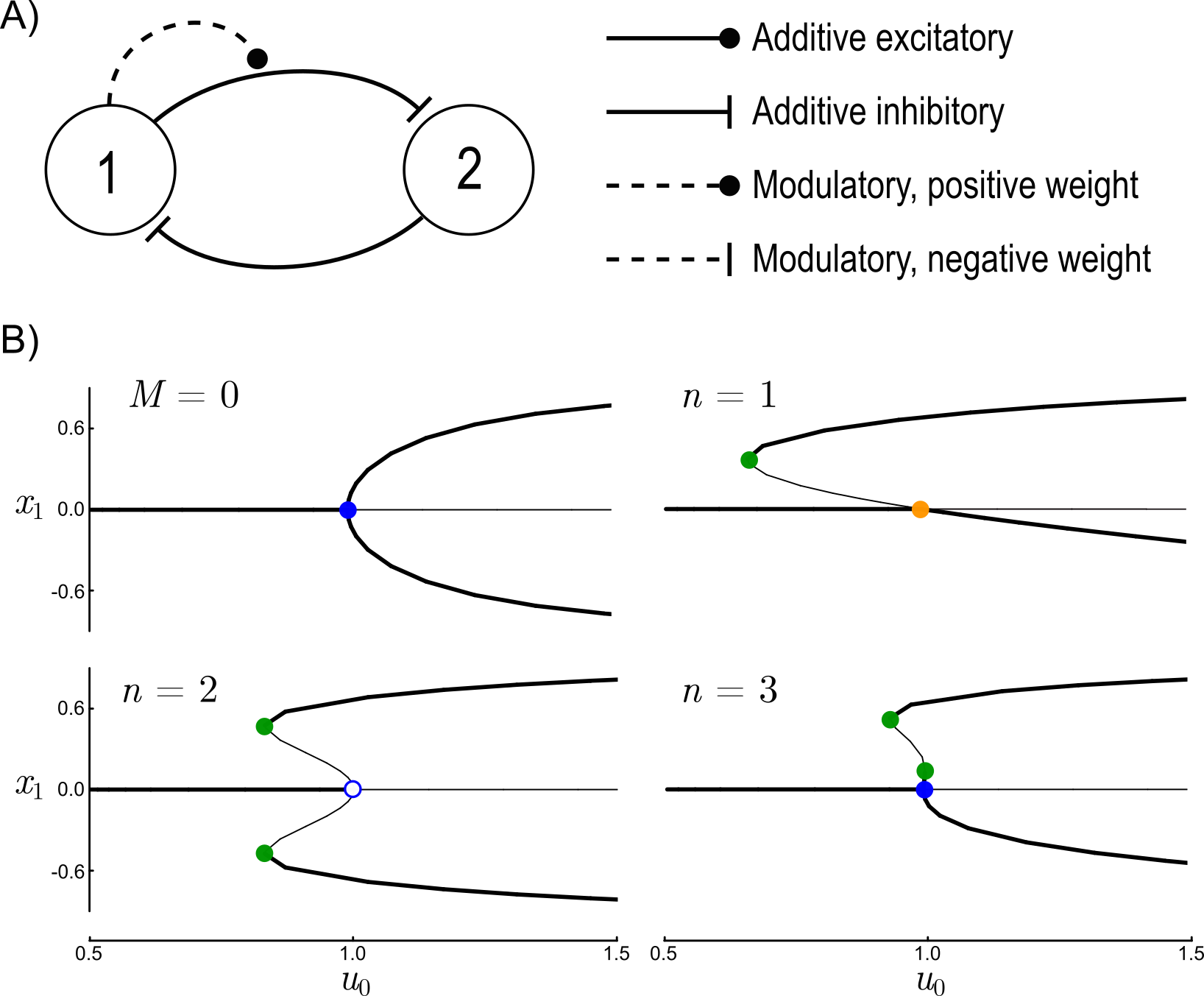}
    \caption{A) Arrows conventions for additive and modulatory interactions (we will use this convention in Figure~\ref{fig: secondary}) and an example of a modulated interaction network for a two-agent, two-options decision-making dynamics~\eqref{eq: bidimensional}. 
    B) Effect of modulatory interactions of various orders on the bifurcation behavior of~\eqref{eq: bidimensional} with additive and modulatory interactions as in A). Stable (unstable) bifurcation branches are indicated by thick (thin) lines. Supercritical pitchfork bifurcations are indicated by blue dots, subcritical pitchfork bifurcations by blue circles,  saddle-node bifurcations by green dots, transcritical bifurcations by orange dots.
    }
    \label{fig: toy}
\end{figure}

Figure~\ref{fig: toy}B shows the bifurcation behavior of model~\eqref{eq: bidimensional} for $b_1=b_2=0$, without modulatory interactions ($M=0$) and with modulatory interactions of different orders ($n=1,2,3$), for additive and modulatory interaction matrices defined in~\eqref{eq: bidimensional matrices}. In all cases, the {\it neutral state} $x_1=x_2=0$ loses stability in a bifurcation at $u_0=1$.

When $M=0$, the model exhibits the symmetric {\it indecision-breaking} or {\it opinion-forming} supercritical {\it pitchfork} bifurcation thoroughly studied  in~\cite{bizyaeva2022nonlinear,leonard2024fast}. For $b_1=b_2=0$, the existence of this symmetric bifurcation arises from the symmetry of the system with respect to swapping agents and swapping options. Agent symmetry is enforced by a symmetric adjacency matrix. Option symmetry is enforced by the odd-symmetry of the vector field, as ensured by using the odd sigmoidal function $S(\cdot)=\tanh(\cdot)$.

When $n=1$, any increase in $x_1$ makes the modulated inhibitory weight $\tilde a_{21}=(u_0+m_{211} x_1)a_{21}$ more negative. That is, the larger $x_1$ is the more it inhibits $x_2$ because $\frac{\partial \tilde a_{21}}{\partial x_1}=m_{211}a_{21}<0$. This modulatory interaction breaks the network symmetry: it favors larger $x_1$ and smaller $x_2$ as compared to the non-modulated regime. Formally, such a modulatory interaction breaks both agent and option symmetries by making the modulated adjacency matrix non-symmetric and the modulated vector field non-odd symmetric. This is reflected in the resulting bifurcation diagram: the pitchfork unfolds into a transcritical bifurcation characterized by a region of bistability between the neutral state and an opinionated state ${\bf x}^*$ characterized by $x_1^*>0>x_2^*$ (the state $x_2$ is not shown in the bifurcation diagrams).

When $n=2$, the modulated inhibitory weight $\tilde a_{21}=(u_0+m_{211} x_1^2)a_{21}$ is non-monotone in $x_1$. If $x_1>0$ ($x_1<0$), an increase in $x_1$ makes $\tilde a_{21}$ more (less) negative because $\frac{\partial \tilde a_{21}}{\partial x_1}=2m_{211}a_{21}x_1<0$ ($>0$). 
This modulatory interaction amplifies mutual inhibition but does {\it not} break network symmetry: it favors larger $x_1$ and smaller $x_2$ for positive $x_1$ (and negative $x_2$) but smaller $x_1$ and larger $x_2$ for negative $x_1$ (and positive $x_2$). The even order of the modulatory interaction preserves the symmetry of the original system in the network equivariant sense of~\cite{golubitsky2003symmetry}.
This is reflected in the resulting bifurcation diagram: the symmetric pitchfork becomes subcritical and is characterized by a region of bistability between the neutral state coexists and two symmetric opinionated states ${\bf x}_{up}^*$ and ${\bf x}_{down}^*$, such that $({\bf x}_{up})_1^*>0>({\bf x}_{up})_2^*$ and $({\bf x}_{down})_1^*<0<({\bf x}_{down})_2^*$.

When $n=3$, the modulated inhibitory weight $\tilde a_{21}=(u_0+m_{211} x_1^3)a_{21}$ is monotone in $x_1$, similar to the case $n=1$, because $\frac{\partial \tilde a_{21}}{\partial x_1}=3m_{211}a_{21}x_1^2\leq 0$. 
As in the case $n=1$, this modulatory interaction breaks both agent and option symmetry and there exists a region of bistability between the neutral state and an opinionated state ${\bf x}^*$ characterized by $x_1^*>0>x_2^*$. However,
for the same modulatory strength $m_{211}$, a cubic modulatory interaction locally preserves the supercritical pitchfork of the non-modulated case, while a linear modulatory interaction unfolds it into a transcritical bifurcation. 


\section{Opinion-forming bifurcations in the presence of modulatory interactions}
\label{sec: analysis}

We start by proving that modulatory interactions do {\it not} change neither the location of the opinion-forming bifurcation nor the associated critical (the Jacobian kernel at bifurcation) and sensitive subspaces (the direction in input space that is amplified nonlinearly along the critical subspace). We then state and interpret our main result that characterizes opinion-forming bifurcation in modulated NOD.

Start by observing that for ${\bf b}={\bf 0}$ the neutral state ${\bf x}={\bf 0}$ is an equilibrium of~\eqref{eq: general} for all $u_0\in\R$.
Let $p_i({\bf x}) = \sum_{j=1}^N a_{ij} \left( u_0 + \sum_{k=1}^N m_{ijk} x_k^n \right) x_j$.
Then~\eqref{eq: general} becomes
$\tau\dot x_i = -x_i+b_i+S(p_i({\bf x})).$
It follows that the Jacobian $J({\bf x})$ of~\eqref{eq: general} at ${\bf x}$ has components
\begin{equation}
    J_{il}({\bf x}) = \left\{ \begin{array}{ll}
        S'\big(p_i({\bf x})\big) \partial p_i/\partial x_l({\bf x}) - 1,& \quad \text{if $l=i$} \\
        S'\big(p_i({\bf x})\big) \partial p_i/\partial x_l({\bf x}), &\quad \text{if $l \neq i$} 
    \end{array}\right.
    \label{eq: jacobian}
\end{equation}
with
\begin{equation}
    \dfrac{\partial p_i}{\partial x_l}({\bf x})\! =\! a_{il}u_0 \!+\! \sum_{k=1}^N a_{il}m_{ilk} x_k^n \!+\! n \sum_{j=1}^N a_{ij} m_{ijl} x_l^{n-1} x_j \, .
    \label{eq: dpdx}
\end{equation}
We have the following evident but key lemma.
\begin{lemma}\label{lem: jacob origin}
    The Jacobian $J({\bf 0})$ of~\eqref{eq: general} at ${\bf x} = \mathbf{0}$ does not depend on the modulatory interaction weights $m_{ijk}$. In particular $J({\bf 0})=-I+u_0S'(0)A$.
    \end{lemma}

The following theorem generalizes~\cite[Theorem~4.2]{bizyaeva2023multi} to model~\eqref{eq: general},\eqref{eq: general vector}. Let $\sigma(A)$ denote the spectrum of $A$.
\begin{theorem}\label{thm: main bifurcation}
    Consider model~\eqref{eq: general vector}. Suppose that $A$ has a strictly leading eigenvalue $\lambda_{max}$, i.e., a simple real eigenvalue satisfying $\lambda_{max}>\max_{\lambda_i\in\sigma(A)\setminus\lambda_{max}}\Re(\lambda_i)$. Let ${\bf v}_{max}$ and ${\bf w}_{max}$ be the right and left eigenvectors associated to $\lambda_{max}$, respectively. Let $u_0^*=(S'(0)\lambda_{max})^{-1}$. Then:\\
    1. (Indecision-breaking bifurcation and critical subspace) For ${\bf b}={\bf 0}$, ${\bf x}={\bf 0}$ is exponentially stable for $u_0<u_0^*$, undergoes a bifurcation at $u_0=u_0^*$, and is unstable for $u_0>u_0^*$. Furthermore, bifurcation branches emanating from $({\bf 0},u_0^*)$ are tangent to ${\bf v}_{max}$.\\
    2. (Sensitivity subspace) If $\langle {\bf b},{\bf w}_{max}\rangle=0$ and $\|{\bf b}\|$ is small enough, there exists a neighborhood $U\ni u_0^*$ such that for all $u_0\in U$ there exists an equilibrium ${\bf x}={\bf x}^*(u_0)$ satisfying $\langle {\bf x}^*(u_0), {\bf v}_{max} \rangle=0$ such that ${\bf x}^*(u_0)$ is stable (unstable) for $u_0<u_0^*$ ($u_0>u_0^*$) and undergoes a bifurcation at $u_0=u_0^*$. If $\langle {\bf b},{\bf w}_{max}\rangle\neq 0$ and $\|{\bf b}\|$ is small enough, then the bifurcation unfolds according to its {\it universal unfolding} (see~\cite[Chapter~4]{Golubitsky1985}).
\end{theorem}
\begin{proof}
    Observe that $J({\bf 0})$ has a simple leading eigenvalue $-1+u_0S'(0)\lambda_{max}$. Let $U$ the matrix that put $J({\bf 0})$ in the Jordan form $$U^{-1}J({\bf 0})U=\begin{pmatrix}
		-1+u0S'(0)\lambda_{max} & {\bf 0}_{1\times N-1}\\
		{\bf 0}_{N-1} & \tilde J({\bf 0})
	\end{pmatrix}.$$ Observe that $U_{\cdot 1}={\bf v}_{max}$ and $U^{-1}_{1\cdot}=({\bf w}_{max})^\top$. Furthermore, all the $N-1$ eigenvalues of $\tilde J({\bf 0})$ have negative real part for $u_0$ sufficiently close to $u_0^*$. Then the first statement follow from Lyapunov's indirect method~\cite[Theorem~4.7]{Khalil2002} and the Center Manifold Theorem~\cite[Theorem~3.2.1]{guckenheimer2013nonlinear}. The second statement follows by applying the Lyapunov-Schmidt reduction~\cite[Section~I.3]{Golubitsky1985} at $({\bf x},u_0)=({\bf 0},u_0^*)$ with respect to right and left singular directions ${\bf v}_{max},{\bf w}_{max}$ and noticing that if $\langle {\bf b},{\bf w}_{max}\rangle=0$ the reduced dynamics along ${\bf v}_{max}$ does not depend on ${\bf b}$, whereas for $\langle {\bf b},{\bf w}_{max}\rangle<0$ the branches predicted by the Lyapunov-Schmidt reduction are predicted by applying unfolding theory~\cite[Chapter~4]{Golubitsky1985} on the resulting scalar equilibrium equation.
\end{proof}

It follows from Theorem~\ref{thm: main bifurcation} that the critical attention value $u_0^*$ is independent of modulatory interactions. It solely depends on the leading eigenstructure of $A$. Theorem~\ref{thm: main bifurcation} also implies that the opinion patterns ${\bf x}_{bif}$ along the opinionated bifurcation branches are solely determined by the right leading eigenvector ${\bf v}_{max}$, i.e., to leading order, ${\bf x}_{bif} = \bar x_{bif} {\bf v}_{max}$, $\bar x_{bif}\in\R$. 
Finally, input sensitivity at bifurcation is also independent of modulatory interactions, as it is determined by the left leading eigenvector ${\bf w}_{max}$.



\section{Shaping global indecision-breaking bifurcation through modulatory interactions}
\label{sec: shaping hierarchical}

In the next section we apply Theorem~\ref{thm: main bifurcation} and illustrate first in a multi-agent, two-option network and then in a single-agent, multi-option network how to analyze and predict the effect of modulatory interactions on the {\it shape} of indecision-breaking bifurcation branches along ${\bf v}_{max}$.

\subsection{Modulated indecision-breaking in a multi-agent, two-option social influence network}

Consider the 5-agent, 2-options, modulated opinion interaction network in Figure~\ref{fig: ring}A. Let all additive links have unitary weight, all modulatory links have weight $\bar m\geq0$, and the order of the modulatory interactions be $n = 1$. We interpret this network as a social network with first-neighbor additive coupling and the presence of an ``influencer'' node (node 1) that affects the network discourse by modulating all additive coupling weights as a function of its state.

In model~\eqref{eq: general vector} the additive interaction matrix $A$ associated to Figure~\ref{fig: ring}A is the adjacency matrix of an undirected ring graph, i.e., the circulant matrix generated by the vector $[0,1,0,0,1]$.
Additionally, because agent 1 is the only modulator and it is modulating all interactions, the modulatory interaction matrix $M$ satisfies $M_{\cdot \cdot 1} = \bar{m}A,\ \bar{m} \in \R$, and 0 elsewhere. Invoking Theorem~\ref{thm: main bifurcation}, there is an opinion-forming bifurcation at $u_0=u_0^*=(S'(0)\lambda_{max})^{-1}$, where $\lambda_{max}=2$ is the largest eigenvalue of $A$. For $S()=\tanh()$, $u_0^*=\tfrac{1}{2}$. $\lambda_{max}=2$ is simple and with right and left eingenvectors proportional to $\mathbf{v}_{max}=\mathbf{w}_{max} = [1,1,1,1,1]^T$. The opinionated bifurcation branches emerging at the opinion-forming bifurcation are tangent to $\mathbf{v}_{max}$, that is, to leading order, they correspond to consensus opinion formation. Similarly, only inputs such that $\langle\mathbf{b},\mathbf{w}_{max}\rangle=\sum_{i=1}^5b_i\neq0$ affect the opinion forming behavior. To focus on the role of the influencer node, we assume ${\bf b}={\bf 0}$ in what follows.

To characterize how the influencer modulation shapes opinionated branches we use the Lyapunov-Schmidt (LS) reduction~\cite[Section~I.3]{Golubitsky1985} with right and left singular eigenvectors $\mathbf{v}_{max},\mathbf{w}_{max}$. Let $v$ be the reduced variable along $\mathbf{v}_{max}$ and $g(v,u_0)$ be the LS reduced equation. Simple computations shows that, at $(v, u_0) = (0, \tfrac{1}{2})$:
\begin{align*}
    g &= \tfrac{\partial g}{\partial v} = \tfrac{\partial g}{\partial u_0} = \tfrac{\partial^2 g}{\partial u_0^2} = 0,\\ 
    \tfrac{\partial  }{\partial v}\tfrac{\partial g}{\partial u_0} &= \tfrac{\partial  }{\partial u_0}\tfrac{\partial g}{\partial v} = 3,\quad
    \tfrac{\partial^2 g}{\partial v^2} = 4\bar m,\quad
    \tfrac{\partial^3 g}{\partial v^3} = -2.
\end{align*}
The recognition problem for the pitchfork~\cite[Prop. 9.2]{Golubitsky1985} then implies that for $\bar m=0$ the opinion forming bifurcation is a supercritical pitchfork (Figure~\ref{fig: ring}B).
When $\bar m>0$, the influencer acts in favor of option 1 and even in the absence of inputs the pitchfork unfolds into a transcritical bifurcation (~\cite[Section~III.7]{Golubitsky1985}). Similarly to Figure~\ref{fig: toy}B ($n=1$), the influencer node modulation of the network discourse creates a pre-bifurcation bistable region in which the group can switch from neutral to option 1 even in the absence of inputs.



\begin{figure}
\vspace{7pt}
    \centering
    \includegraphics[width = 0.4\textwidth]{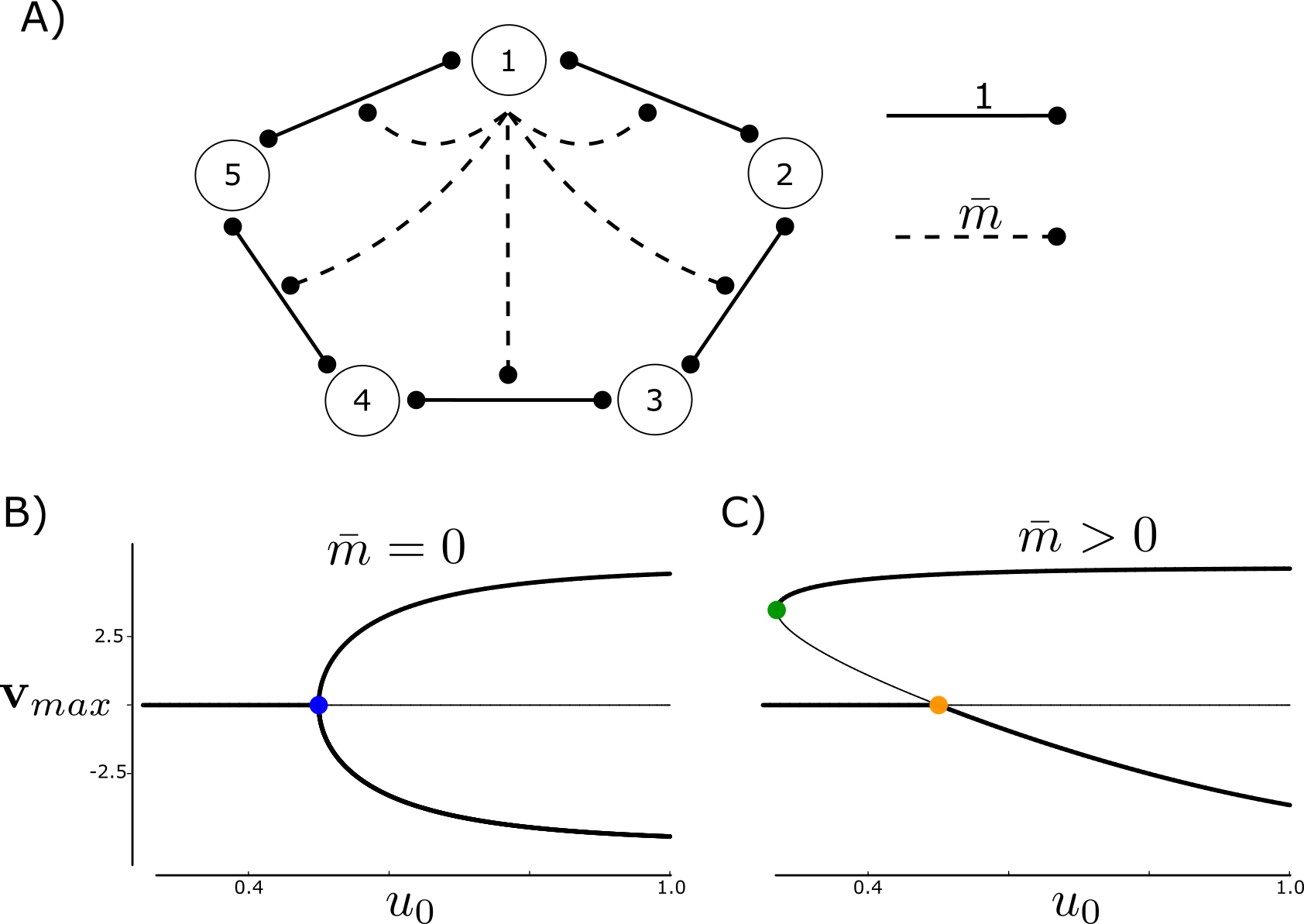}
    \caption{ A) The modulated opinion interaction network under study. B) Bifurcation diagram in the absence of inputs and $m=0$. The model undergoes a supercritical pitchfork at $u_0=\tfrac{1}{2}$ along the consensus subspace generated by ${\bf v}_{max}$. C) Bifurcation diagram for $m = 0.5$. The pitchfork unfolds into a transcritical bifurcation (yellow dot). The associated bistability region reflects the effects of the influencer node in shaping the network discourse toward option 1.}
    \label{fig: ring}
\end{figure}

\subsection{Modulated conditional decision-making in a single-agent multi-option network for robot navigation}
\label{sec: hierarchical decisions}

The single-agent, multi-option, modulated opinion interaction network in Figure~\ref{fig: secondary}A provides an extension of the NOD-based robot navigation controller introduced in~\cite{cathcart2023proactive}. It consists of two mutual inhibitory NOD subnetworks made by nodes $\{1,2\}$, `drive or stay' , and $\{3,4\}$, `steer left or steer right', respectively. The two subnetworks are disconnected at the additive interaction level but node 1 of subnetwork $\{1,2\}$ modulates the mutual inhibition strength of subnetwork $\{3,4\}$ with modulatory weight $\bar m$. The analysis in the remainder of the section is illustrated by Figure~\ref{fig: secondary}B.

The mutual inhibition strength of subnetwork $\{1,2\}$ is $\alpha$. Invoking Theorem~\ref{thm: main bifurcation}, for $u_0<\alpha^{-1}$ the neutral state $x_1,x_2=0$ (idle state) is stable. For $u_0>\alpha^{-1}$ the neutral state becomes unstable in a pitchfork bifurcation giving rise to two opinionated stable equilibria characterized by $x_1>0>x_2$ (`drive') and $x_1<0<x_2$ (`stay'), respectively.

The mutual inhibition strength of subnetwork $\{3,4\}$ is $\beta+\bar m x_1$, so in this subnetwork the opinion-forming pitchfork bifurcation happens for $u_0=(\beta+\bar m x_1)^{-1}$. Hence, if $\bar m>0$, the more positive (negative) $x_1$ is, the smaller (larger) the basal attention $u_0$ needed for the opinion-forming pitchfork of subnetwork $\{3,4\}$ to happen is. For $u_0<(\beta+\bar m x_1)^{-1}$ the neutral state $x_3,x_4=0$ (`no steering') is stable. For $u_0>(\beta+\bar m x_1)^{-1}$ the neutral state is unstable and two stable opinionated equilibria appear, characterized by $x_3>0>x_4$ (`steer left') and $x_3<0<x_4$ (`steer right'), respectively.

In the network model of Figure~\ref{fig: secondary}A it is natural to assume $\alpha>\beta$ in such a way that, as attention increases, the decision to drive (or stay) happens before the decision to steer. If $\bar m=0$, the steering pitchfork of subnetwork $\{3,4\}$ happens for the same basal attention value $u_0=\beta^{-1}$, independently of the decision state of the drive-or-stay subnetwork $\{1,2\}$. Conversely, for $\bar m>0$ and recalling that $x_1>0$ ($x_1<0$) along the driving (staying) branch, the critical attention value to trigger a steering decision is lower (higher) when subnetwork $\{1,2\}$ is in the drive (stay) state. This ensures faster and more sensitive steering decisions when driving as compared to staying, a desirable property for performing and efficient navigation.

The modulation of the basal attention needed to trigger a decision as a function of another decision outcome can be understood as a form of soft or flexible conditional decision-making: the sensitivity of a subordinate decision (here steering) is conditioned to the outcome of a primary decision (here drive-or-stay).
The continuous state and parameter nature of  modulated NOD, and the organizing role of its bifurcations, makes this kind of soft conditional decision-making tunable and adaptable according to the principles of fast-and-flexible decision-making~\cite{leonard2024fast}.

\begin{figure}
\vspace{7pt}
    \centering
    \includegraphics[width = 0.43\textwidth]{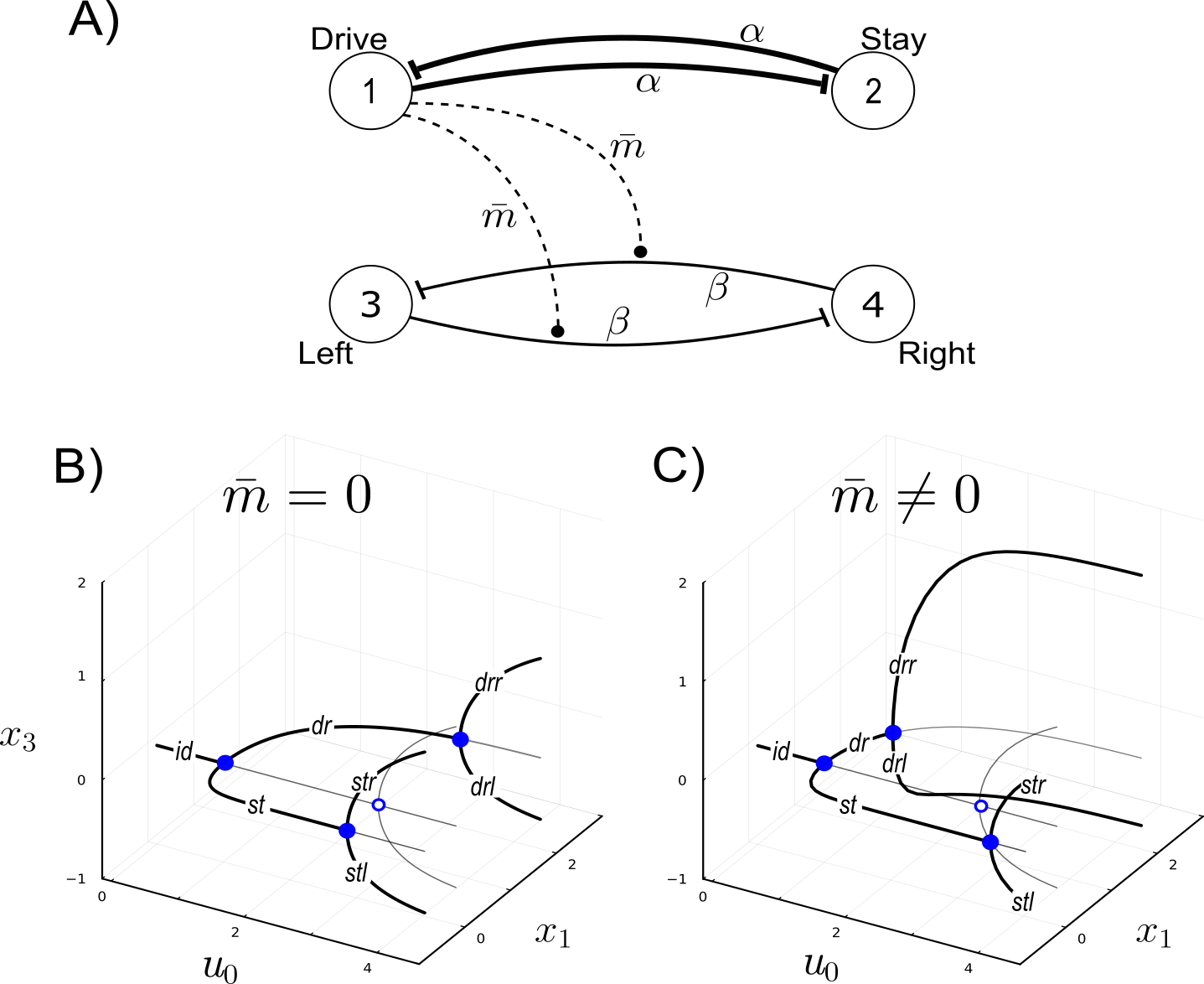}
    \caption{A) The modulated opinion interaction network under study. B) Bifurcation diagram for $\alpha=1,\beta=0.3,\bar m=0$. Branches are labelled with the associated robot navigation commands ($\mathsf{id}$: idle; $\mathsf{dr}$: drive; $\mathsf{st}$: stay; $\mathsf{drr}$: drive+steer right; $\mathsf{drl}$: drive+steer left; $\mathsf{str}$: stay+steer right; $\mathsf{stl}$: stay+steer left)  C) Same as B but for $\bar m=2$. In the presence of modulatory interactions, the drive+steer bifurcation ($\mathsf{drr}$,$\mathsf{drl}$ branches) happens for lower $u_0$ than the stay+steer bifurcation ($\mathsf{str}$,$\mathsf{stl}$ branches).}
    \label{fig: secondary}
\end{figure}

\section{Conclusions}

We introduced modulatory, that is, multiplicative, interactions in nonlinear opinion dynamics (NOD) of fast-and-flexible decision-making, characterized their effects on opinion-forming behaviors, and illustrated their use for the design of fast-and-flexible {\it conditional} decision-making. The proposed model is inspired by widespread evidence of the role of modulatory interactions in biological (neural, molecular) and social decision-making networks. The chosen parametrization makes the model tractable and we showed how to rigorously characterize the modulated NOD opinion-forming bifurcation behavior. In particular, the possible opinion patterns emerging at a modulated opinion-forming bifurcation can be characterized in full generality and coincide with the non-modulated case. Similarly for the input patterns that are amplified at the opinion-forming bifurcation. However, the shape of the modulated opinion-forming bifurcation branches, and therefore the resulting opinion-forming behaviors, change markedly in the presence of modulatory interactions. We showed how the Lyapunov-Schmidt reduction method can be used to predict this shaping effect in a social influence network. Finally, we illustrated how modulatory interactions can be used to augment a NOD model for robot navigation to exhibit conditional decision-making in which a steering decision is conditioned to a drive-or-stay decision. To summarize, the proposed modulated NOD retain the mathematical tractability but also largely enriches the behavior and the possible applications of the original NOD .

\bibliographystyle{IEEEtran}
\bibliography{biblio.bib}

\end{document}